\documentclass[reqno,11pt,centertags]{article}
\usepackage{amsmath,amsthm,amscd,amssymb,latexsym,upref}
\date{\today}                            

\input epsf
\usepackage{epsfig}
\usepackage[T2A,OT1]{fontenc}
\usepackage[ot2enc]{inputenc}
\usepackage[russian,english]{babel}
\usepackage[margin=3.5 cm]{geometry}
\usepackage{epic,eepicemu}
\usepackage[all]{xy}

\usepackage{hyperref}

\newcommand{\bbD}{{\mathbb{D}}}

\newcommand{\bbR}{{\mathbb{R}}}

\newcommand{\bbZ}{{\mathbb{Z}}}
\newcommand{\bbC}{{\mathbb{C}}}

\DeclareMathAlphabet{\mathpzc}{OT1}{pzc}{m}{it}

\newcommand{\cA}{{\mathcal{A}}}

\newcommand{\cC}{{\mathcal{C}}}
\newcommand{\cD}{{\mathcal{D}}}
\newcommand{\cE}{{\mathcal{E}}}
\newcommand{\cF}{{\mathcal{F}}}
\newcommand{\cG}{{\mathcal{G}}}

\newcommand{\cJ}{{\mathcal{J}}}
\newcommand{\cK}{{\mathcal{K}}}

\newcommand{\cM}{{\mathcal{M}}}
\newcommand{\cN}{{\mathcal{N}}}

\newcommand{\cS}{{\mathcal{S}}}
\newcommand{\cT}{{\mathcal{T}}}
\newcommand{\cU}{{\mathcal{U}}}
\newcommand{\cV}{{\mathcal{V}}}
\newcommand{\cW}{{\mathcal{W}}}

\newcommand{\sE}{{\mathsf{E}}}

\newcommand{\fA}{{\mathfrak{A}}}
\newcommand{\fB}{{\mathfrak{B}}}

\newcommand{\fa}{{\mathfrak{a}}}

\newcommand{\fc}{{\mathfrak{c}}}
\newcommand{\fe}{{\mathfrak{e}}}
\newcommand{\ff}{{\mathfrak{f}}}

\newcommand{\fj}{{\mathfrak{j}}}

\newcommand{\fm}{{\mathfrak{m}}}

\newcommand{\e}{{\epsilon}}
\newcommand{\ve}{{\varepsilon}}
\renewcommand{\k}{\varkappa}

\newcommand{\z}{\zeta}

\newcommand{\pd}{{\partial}}

\def\restr#1{\,\vrule\,\lower1ex\hbox{$#1$}}

\def\a{\alpha}
\def\b{\beta}

\def\e{\epsilon}

\def\g{\gamma}

\def\k{\kappa}
\def\vk{\varkappa}
\def\l{\lambda}
\def\u{\upsilon}
\def\U{\Upsilon}

\def\o{\omega}
\def\O{\Omega}

\def\s{\sigma}
\def\S{\Sigma}
\def\t{\theta}

\def\z{\zeta}
\def\R{{\bf R}}

\renewcommand{\Im}{\text{\rm Im}\,}

\allowdisplaybreaks \numberwithin{equation}{section}
\newtheorem{theorem}{Theorem}[section]

\newtheorem{lemma}[theorem]{Lemma}
\newtheorem{proposition}[theorem]{Proposition}

\newtheorem{corollary}[theorem]{Corollary}

\theoremstyle{definition}
\newtheorem{definition}[theorem]{Definition}

\newtheorem{remark}[theorem]{Remark}

\date{\today}

\author{P. Yuditskii}

\title{Direct Cauchy Theorem and Fourier integral \\ in Widom domains}

\begin{document}
\maketitle

\begin{center}\hskip 2 cm
{\textit{In the study of mathematics, there is a grave injustice: \\ \hskip 3 cm we put in so much effort, but we get such miserable results...\\
\hskip 5 cm Larry Zalcman (from a private conversation)}}
\end{center}

\vskip 1 cm

\begin{abstract} We derive Fourier integral associated to the complex Martin function in the Denjoy domain of Widom type with the Direct Cauchy Theorem (DCT). As an application we study reflectionless Weyl-Titchmarsh functions in such domains, related to them canonical systems and transfer matrices.
The DCT property appears to be crucial in many aspects of the underlying theory.
\end{abstract}

\section{Introduction}
We develop here some specific aspects of the general de Branges theory \cite{BR}, which deals with the
 function theory in infinitely connected domains \cite{Has83} and spectral properties of random and almost-periodic operators
\cite{PaFi92}.

Let $\sE$ be a closed unbounded subset of the positive half axis,
$$
\mathsf{E}=\bbR_+\setminus\cup_{j\ge 1}(a_j,b_j).
$$
We assume that the domain $\O=\bbC\setminus \sE$ is regular in the sense of the potential theory \cite{GM}. By $\cG(\l,\l_0)$ we denote the Green function of the domain with singularity at 
$\l_0\in\O$. The \textit{complex Green function} is defined by
$$
\Phi_{\l_0}(\l)=e^{ i\t_{\l_0}(\l)}, \quad \t_{\l_0}(\l)=-\star \cG(\l,\l_0)+i\cG(\l,\l_0),
$$
where $\star \cG(\l,\l_0)$ is the harmonically  conjugated to $\cG(\l,\l_0)$ function,
$\star \cG(\l_*,\l_0)=0$ for a  normalization point $\l_*\in \bbR_-$. The complex Green function is multivalued in $\O$. Let $\pi_1(\O)$ be the fundamental group of
$\O$. It is generated by simple loops $\{\g^{(j)}\}_{j\ge 1}$, $\g^{(j)}$  starts and ends at $\l_*\in \bbR_-$ and goes through the gap $(a_j,b_j)$. To be extended by continuity along $\g^{(j)}$ the complex Green function obeys the following identity
$$
\Phi_{\l_0}(\g^{(j)}(\l))=e^{2\pi i\o(\l_0,E_j)}\Phi_{\l_0}(\l),
$$
where $\o(\l_0,\sE_j)=\o(\l_0,\sE_j,\O)$ is the harmonic measure of the set $\sE_j=\sE\cap[0,a_j]$ computed at $\l_0$. 

By $\pi_1(\O)^*$ we denote the group of  characters of the group $\pi_1(\O)$, see e.g. \cite{HeRo79},
$$
\a: \pi_1(\O)\to \bbR/\bbZ, \quad \a(\g_1\g_2)=\a(\g_1)+\a(\g_2),\ \g_j\in\pi_1(\O).
$$ 
We say that $F(\l)$ is character automorphic with a certain character $\a\in \pi_1(\O)^*$ if 
\begin{equation}\label{3oct1}
F(\g(\l))=e^{2\pi i\a(\g)} F(\l).
\end{equation}
Note that $|F(\l)|$ is single valued in the domain.

For a fixed character $\a$ by $H^\infty_{\O}(\a)$ we denote the collections of bounded analytic multivalued functions $F(\l)$ such that \eqref{3oct1} holds \cite{Has83}.
 More generally the Hardy spaces $H^p_{\O}(\a)$ are formed by functions which obeys \eqref{3oct1} and $|F(\l)|^p$ possesses a 
harmonic majorant in the domain.

\begin{theorem}[Widom]
The following two statements are equivalent
\begin{itemize}
\item $H_\O^{\infty}(\a)$ contains a non constant function for all $\a\in\pi_1(\O)^*$.
\item Let $\{c_j\}$ be the collection of critical points of $\cG(\l,\l_*)$, i.e., $\nabla \cG(c_j,\l_*)=0$. Then
\begin{equation}\label{3oct2}
\sum \cG(c_j,\l_*)<\infty.
\end{equation}
\end{itemize}
\end{theorem}

In the Widom domain $\O$ the harmonic measure $\o(\l_*,d\xi)$ is absolutely continuous with respect to the Lebesgue measure, moreover 
\begin{equation}\label{3oct3}
\o(\l_*,d\xi)=|\Psi_{\l_*}(\xi)|\frac{d\xi}{\sqrt{\xi}}, \ \xi\in E,
\end{equation}
where  $\Psi(\l)=\Psi_{\l_*}(\l)$ is an \textit{outer} character automorphic function,
$\Psi(\g(\l))=e^{2\pi i \b_\Psi(\g)}\Psi(\l)$. Using this function we can reduce Hardy spaces $H_\O^p(\a)$ to the Smirnov spaces $E_\O^p(\b)$.

\begin{definition}
We say that $F(\l)$ belongs to the Smirnov class $N_+(\O)$ if it can be represented as a ratio of two bounded character automorphic functions with an outer denominator \cite[Ch. II, Sect. 5]{Gar07}.
We say that $F(\l)\in N_+(\O)$ belongs to the class $E_\O^p(\a)$ if \eqref{3oct1} holds, and its boundary values ($\xi\pm i0$, $\xi\in \sE$)  satisfy
\begin{equation}\label{4oct1}
\frac 1{2\pi}\oint_\sE |F(\xi)|^p\frac{d\xi}{\sqrt\xi}:=\frac 1 {2\pi}\int_\sE (|(F(\xi+i0)|^p+|F(\xi-i0)|^p)\frac{d\xi}{\sqrt\xi}<\infty.
\end{equation}
\end{definition}

\begin{proposition}
$F(\l)$ belongs to $E_\O^p(\a)$ if and only if 
$$
\Psi^{-1/p}_{\l_*}(\l) F(\l)\in H^2_{\O}(\a-\b_{\Psi_{\l_*}^{1/p}}), \quad \Psi^{1/p}_{\l_*}(\g(\l))=
\exp{2\pi i\b_{\Psi_{\l_*}^{1/p}}(\g)}\,\Psi^{1/p}_{\l_*}(\l).
$$
\end{proposition}

Let $\fj$ be the character generated by the function $\sqrt\l$, i.e., $e^{2\pi i\fj(\g^{(m)})}=-1$ for all generators $\g^{(m)}$.

\begin{definition} 
We say that the Direct Cauchy Theorem (DCT) holds in $\O$ if
\begin{equation}\label{4oct2}
\frac 1{2\pi i}\oint_\sE F(\xi)\frac{d\xi}{\sqrt\xi}=0,\quad \forall F\in E^1_{\O}(\fj).
\end{equation}
\end{definition}

The space $E_\O^2(\a)$ possesses the reproducing kernel, which we denote by $k^\a_{\l_0}(\l)=k^\a(\l,\l_0)$,
\begin{equation}\label{4oct2}
\langle F,k^\a_{\l_0}\rangle=\frac 1{2\pi}\oint_\sE\overline{k^\a(\xi,\l_0)} F(\xi)\frac{d\xi}{\sqrt{\xi}},\quad\forall F\in E^2_{\O}(\a).
\end{equation}

Let $F(\l)$ be a measurable function on $\pd\O$, i.e., $F=\{ F(\xi+i0), F(\xi-i0)\}_{\xi\in \sE}$. We say that $F\in L^p_{\pd\O}$ if 
\eqref{4oct1} holds.

Let $\cW_{\a}(\l)$ be a solution of the following extremal problem
$$
\cW_\a(\l_*)=\sup\{ W(\l_*):\ W(\l)\in H_\O^\infty(\a),\ \|W\|\le 1\}.
$$
The minimizer exists due to the compactness arguments.

\begin{theorem}[see \cite{Has83}] \label{thdct}
In a Widom domain $\O$ the following are equivalent 
\begin{itemize}
\item[(i)]
DCT holds.
\item[(ii)] $k^\a(\l_0,\l_0)$ is a continuous function in $\a\in\pi_1(\O)^*$.
\item[(iii)] $\cW_{\a}(\l)\to 1$ for a fixed $\l\in\O$ and $\a\to 0_{\pi_1(\O)^*}$.
\item[(iv)] $F\in  L^2_{\pd\O}\ominus E^2_{\O}(\a)$ if and only if $\overline{F}\in E^2_{\O}(\fj-\a)$ for all $\a\in\pi_1(\O)^*$.
\end{itemize}
\end{theorem}

These special functions in a Widom domain (i.e., the complex Green functions and reproducing kernels, which can be given in terms of canonical products, see Sect. 2) allow to construct intrinsic basis in the Hardy/Smirnov spaces of character automorphic functions.

\begin{theorem}[see \cite{SY97}]
Let $\b_0$ be the character generated by the complex Green function $\Phi_{\l_0}$. The following system of functions
\begin{equation}\label{4oct3}
e_n^\a(\l)=e_n^\a(\l,\l_0)=\Phi_{\l_0}(\l)^n \frac{k^{\a-n\b_0}_{\l_0}(\l)}{\sqrt{k^{\a-n\b_0}(\l_0,\l_0)}}
\end{equation}
forms an orthonormal basis in $E^2_{\O}(\a)$. That is, for an arbitrary $F\in E^2_{\O}(\a)$
\begin{equation}\label{4oct4}
F(\l)=\sum_{n\ge 0} c_n e^\a_n(\l), \quad c_n=\langle F, e_n^\a\rangle.
\end{equation}
\end{theorem}

Our first goal is to prove a continual analog of the decomposition \eqref{4oct4}. First of all we introduce the limit counterpart  of the Green function associated to a boundary point of the domain. We choose infinity as such point (this explains why we were interested to have $\sE$ as an unbounded set).

The Martin function $\cM(\l)$ in $\O$ (with respect to $\infty$) is a positive harmonic function continuously vanishing at all boundary points of the domain except for $\infty$,  especially for Denjoy domains see e.g. \cite{Koo1}. This function is unique up to a positive multiplier and can be obtain in the following limit procedure
\begin{equation}\label{4oct5}
\cM(\l)=\lim_{\l_0\to-\infty}\frac{\cG(\l,\l_0)}{\cG(\l_*,\l_0)}.
\end{equation}
By $\l_0\to-\infty$ we mean $\l_0\in\bbR_-\subset\O, \l_0\to\infty$.
Evidently, in this case the Martin function meets the normalization $\cM(\l_*)=1$. Respectively the complex Martin function is given as
\begin{equation}\label{84oct5}
e^{i\t(\l)}=\lim_{ \l_0\to-\infty}e^{i\t_{\l_0}(\l)/\cG_{\l_0}(\l_*)}, \quad \Im \t(\l)=\cM(\l).
\end{equation}
In this case  $e^{ix\t(\l)}$ is a character automorphic function  for an arbitrary real $x$. We also introduce a special notation for the corresponding character
\begin{equation}\label{4oct6}
e^{ix\t(\g(\l))}= e^{2\pi i x\eta(\g)}e^{i x\t(\l)}.
\end{equation}

Therefore the system of subspaces $e^{ix\t}E^2_{\O}(\a-\eta x)$, $x\in \R_+$, is a natural continuous counterpart of the generating the Fourier series \eqref{4oct4}  discrete system 
of subspaces $\Phi_{\l_0}^n E_{\O}^2(\a-n\b_0)$, $n\in\bbZ_+$.
Our first main result is the following theorem.
\begin{theorem}\label{thm1}
Let  $\O$ be a Widom domain with DCT. 
The following limit exists
\begin{equation}\label{pat8oct2}
v_{\a}(\l)=v_{\a,\l_*}(\l)=\lim_{\l_0\to-\infty}\frac{k^\a(\l,\l_0)}{k^{\a}(\l_*,\l_0)}
\end{equation}
 and represents a continuous  function in $\a\in\pi_1(\O)^*$.
We define a positive continuous measure by its distribution function
\begin{equation}\label{pat8oct3}
\vk_{\l_*}^\a(x)=k^{\a}(\l_*,\l_*)-e^{-2x\Im\t_*}k^{\a-\eta x}(\l_*,\l_*), \quad \t_*=\t(\l_*).
\end{equation}
Then the following (Fourier) transform 
\begin{equation}\label{m8oct2}
(\cF^\a f)(\l)=\int_{-\infty}^\infty
f(x)
e^{i x(\t(\l)-\t_*)}v_{\a-\eta x,\l_*}(\l)d\vk_{\l_*}^\a(x), \quad \l\in\pd\O,
\end{equation}
acts unitary from $L^2_{d\vk^{\a}_{\l_*}}$ to $L^2_{\pd\O}$. Moreover   $\left.\cF^\a L^2_{d\vk^{\a}_{\l_*}}\right |_{\bbR_+}=E^2_{\O}(\a)$.
\end{theorem} 

We apply this result to introduce and study associated to such domains reflectionless Weyl-Titchamrsh functions, canonical systems,   and 
transfer matrix functions.

The Nevanlinna  class  is formed by  functions $w(\l)$ analytic in the upper half $\bbC_+$ and having positive imaginary part, $\Im w(\l)>0$.  Such functions possess the additive
\begin{equation}\label{11oct7}
w(\l)=a\l+b+\int_{\bbR}\frac{1+\xi\l}{\xi-\l}\frac{d\sigma(\xi)}{1+\xi^2},\quad a>0,\ b\in\bbR,
\end{equation}
and the multiplicative
\begin{equation}\label{11oct8}
w(\l)=c\, e^{\int_{\bbR}\frac{1+\xi \l}{\xi-\l}\frac{\chi(\xi)d\xi}{1+\xi^2}}, \quad
c>0,
\end{equation}
representations. Here $\sigma$ is a nonnegative measure such that
$$
\int_{\bbR}\frac{d\sigma(\xi)}{1+\xi^2}<\infty,
$$
and a measurable function $\chi(\xi)$ is such that $\chi(\xi)\in [0,1]$. Moreover
$$
\chi(\xi)=\frac 1 \pi \arg w(\xi+i0), \quad \text{a.e.}\ \xi\in\bbR.
$$
We say that $w$ belongs to the Stieltjes class $\cS$ if the measure $\sigma$ in the representation \eqref{11oct7} is  supported on the positive half axis.
Such functions allow an analytic extension in the lower half plane by the symmetry principle, $w(\l)=\overline{w(\overline{\l})}$ (the function is analytic in
$\bbC\setminus\bbR_+$). 

By $\cS_0$ we denote the subclass $\cS$ of functions $m(\l)$ such that
$$
\lim_{\l\to-0} m(\l)=0.
$$
These functions possess a special additive representation
\begin{equation}\label{115oct7}
m(\l)=a\l+\int_{\bbR_+}\frac{\l d\sigma(\xi)}{\xi -\l},\quad a>0,\ \int_{\bbR_+}\frac{d\s(\xi)}{1+\xi}<\infty, \ \s\{0\}=0.
\end{equation}
Note $\cS_0$ is related to the Nevanlinna class functions $n(\mu)$  with an associated  symmetric measure in a simple way
\begin{equation}\label{15oct3}
n(\mu)=\frac 1 \mu m(\mu^2)=a\mu+\frac 1 2\int_{\bbR_+}
\left\{\frac{1}{t-\mu} -\frac{1}{t+\mu}\right\}d\sigma(t^2), \quad \mu\in\bbC_+.
\end{equation}

\begin{definition}
We say that $m_+\in\cS_0$ belongs to the set $m_0(\sE)$ if  there exists $m_-$ of the Stieltjes class such that 
\begin{equation}\label{15oct4}
m_-(\l)=-\overline{m_+(\l)}\quad\text{ for a.e. $\l\in\sE$,}
\end{equation}
 and the following two their symmetric combinations
\begin{equation}\label{11oct6}
R_0(\l)=-\frac{1}{m_+(\l)+m_-(\l)},\quad R_1(\l)=\frac{m_+(\l)m_-(\l)}{m_+(\l)+m_-(\l)}
\end{equation}
are holomorphic in $\O=\bbC\setminus \sE$.
\end{definition}

The relation \eqref{15oct4} means that $m_+(\l)$ has a \textit{pseudocontinuation} \cite[Lecture II, Sect. 1]{Nik} through the set $\sE$. In the spectral theory it is called the \textit{reflectionless} property \cite{Rem11, PoRem}. Due to this property   $R_i(\l)$  assumes  pure imaginary boundary values a.e. on $\sE$ (on the negative half axis and in the gaps $(a_j,b_j)$  they are real valued by the definition).

The following proposition gives a parametric description of the class 
$$m_0(\sE)\simeq\bbR_+\times\pi_1(\O)^*.$$
\begin{theorem}\label{th9}
Let $\O$ be of Widom type and DCT hold. Then $m_+\in m_0(\sE)$ if and only if it is of the form
\begin{equation}\label{15oct5}
m_+(\l)=\frac{m_+(\l_*)}{i\sqrt{\l_*}}\fm_+^\a(\l),\quad \fm_+^\a(\l):=i\sqrt{\l}\frac{v_{\a+\fj}(\l)}{v_{\a}(\l)},
\end{equation}
where $\a\in\pi_1(\O)^*$. 
\end{theorem}

We describe the collection $\{\fm_+^\a(\l)\}_{\a\in\pi_1(\O)^*}$ as the Weyl-Titchamrsh functions of canonical systems.

\begin{theorem}\label{thm2} The following limit exists
\begin{equation}\label{15oct10}
\frac{\U^\a(x)}{\U^\a(0)}:=\lim_{\l\to - 0}\frac{v_{\a-\eta x}(\l)}{v_{\a}(\l)}
=\lim_{\l\to - 0}\lim_{\l_0\to - \infty}\frac{k^{\a-\eta x}(\l,\l_0)k^{\a}(\l_*,\l_0)}{k^{\a-\eta x}(\l_*,\l_0)k^{\a}(\l,\l_0)}
\end{equation}
and can be given explicitly as
\begin{align}\label{28sep1t}
\nonumber
\U^\a(x)&={\U_{\l_*}^\a(x)}=
\sqrt{k^{\a-\eta x}(\l_*,\l_*)+k^{\a+\fj-\eta x}(\l_*,\l_*) }\\
&\times  \exp\frac 1 2\int_0^ x \frac{d\, e^{-2\xi\Im\t_*}\left(k^{\a+\fj-\eta\xi}(\l_*,\l_*)-k^{\a-\eta \xi}(\l_*,\l_*) \right)}{ e^{-2\xi\Im\t_*}
\left(k^{\a+\fj-\eta\xi}(\l_*,\l_*)+k^{\a-\eta \xi}(\l_*,\l_*) \right)}.
\end{align}
Let
$$
\cJ=\begin{bmatrix}0&1\\-1&0
\end{bmatrix},\quad \cT_\a(x,\l)=\begin{bmatrix}\tau^{\a+\fj}(x) &0\\ 0&\l\tau^\a(x)\end{bmatrix},
\quad \tau^\a(x)=\frac i{\sqrt{\l_*}}\int_0^x\frac{e^{2\xi\Im\t_*}d\vk^{\a}(\xi)}{\U^\a(\xi)^2},
$$
 and $\fA_\a(\l,x)$ be the family of the transfer matrices of the canonical system given in the integral form
\begin{equation}\label{15oct8}
\fA_{\a}(\l,x)\cJ=\cJ-\int_0^x\fA_\a(\l,\xi)d\cT_\a(\l,\xi),\quad
\fA(x,\l)=\begin{bmatrix}\fa^\a_{11}&\fa^\a_{12}\\
\fa^\a_{21}&\fa^\a_{22}
\end{bmatrix}(\l,x).
\end{equation}
Then $\fm^\a_+(\l)$, defined in \eqref{15oct5}, is the  corresponding Weyl-Titchamrsh function. That is, for an arbitrary $\l\in\O$, its value is the unique intersection point of the nesting Weyl circles
\begin{equation}\label{15oct9}
\fm^\a_+(\l)=\lim_{x\to\infty}\frac{\fa^\a_{22}(\l,x)w-\fa^\a_{21}(\l,x)}{-\fa^\a_{12}(\l,x)w+\fa^\a_{11}(\l,x)},\quad 
w\in\bbR_+\cup\{\infty\}.
\end{equation}
\end{theorem}

Being objects of the general de Branges theory, the transfer matrices have  some standard properties. They form a monotonic family of \textit{entire}  matrix functions $\cJ$-contractive in the upper half plane,
$$
\frac{\cJ-\fA_\a(\l,x)\cJ\fA(\l,x)^*}{\l-\overline{\l}}\ge 0,
$$
also $\overline{\fA(\bar \l,x)}=\fA(\l,x)$, $\det\fA(\l,x)=1$.
Note that passing from the class $\cS_0$ to the class of symmetric Nevanlinna functions, see \eqref{15oct3}, we pass to the family of transfer matrices
$$
\fB(\mu,x)=\begin{bmatrix}\mu&0\\0 &1\end{bmatrix}\fA(\mu^2,x)
\begin{bmatrix}1/\mu&0\\0 &1\end{bmatrix},
$$
which satisfy the canonical system
with a diagonal Hamiltonian, i.e.,
\begin{equation*}
\fB_{\a}(\mu,x)\cJ=\cJ-\mu \int_0^x\fB_\a(\mu,\xi)\begin{bmatrix}d\tau^{\a+\fj}(\xi) &0\\
0&d\tau^\a(\xi)
\end{bmatrix}.
\end{equation*}
The featured  properties are given in the following statement. 

\begin{corollary}\label{c15oc1}
1. To find the argument $x$ of the entire matrix function $\fA(\l,x)$ we consider its entries as functions of bounded characteristic  in the domain $\O$,
after that we  compute the ``exponential type" of these functions, i.e.,
\begin{equation}\label{15oc1}
x=\lim_{\l\to-\infty}\frac{\log \|\fA(\l,x)\|}{\cM(\l)}.
\end{equation}
2. Assume that
\begin{equation}\label{b13apr3}
\lim_{\l\to-\infty}\frac{\cM(\l)}{\sqrt{|\l|}}=0.
\end{equation}
Then, the generating canonical system measures $\tau^\a$ and $\tau^{\a+\fj}$ are mutually singular.
\end{corollary}

\begin{remark}
A set $\sE$ is called of Akhiezer-Levin type if
\begin{equation*}
\lim_{\l\to-\infty}\frac{\cM(\l)}{\sqrt{|\l|}}>0.
\end{equation*}
Such kind of domains are widely studied in particular  in a connection with 1-D Shr\"odinger operators (i.g. the Marchenko-Ostrovskii class
\cite{Mar}, see also \cite{SY95}).
A quite general result, including the KdV hierarchy, was presented recently in \cite{EVY}. That is why we are mostly interested in the case \eqref{b13apr3}. As the simplest example, one can have in mind a set which is formed by geometric progressions
$a_n=\rho^n a_0$, $b_n=\rho^n b_0$, $n\in\bbZ$, $0<a_0<b_0<\rho a_0$.
\end{remark}

\begin{remark} Let us say that 
0 is a regular point  for $\sE$ if the following limit
$$
\lim_{\l\to-0}\frac{v_{\b}(\l)}{v_{\a}(\l)}
$$
exists for all $\a,\b\in\pi_1(\O)^*$. Then we can define
$$
\Xi(\a)=\lim_{\l\to-0}\frac{v_{\a}(\l)}{v_{0}(\l)}, \quad 0=0_{\pi_1(\O)^*},
$$
and represent the limit \eqref{15oct10} in terms of this function on the group of characters,
$$
\frac{\U^{\a}(x)}{\U^{\a}(0)}=\frac{\Xi(\a-\eta x)}{\Xi(\a)}.
$$
Especially in the finite gap case one gets explicit formulas in terms of theta functions \cite{MumT2}.
However, say, for a geometric progression set $\sE$, 0 is not a regular point, nevertheless 
$\U_{\a}(x)$ has perfect sense and represents a continuos function in $x$. On the other hand regularity of 0 
is, of course, not an extraordinary property of a set $\sE$. The simplest case of regularity: $[0,\ve]\subset\sE$ for some $\ve>0$.
\end{remark}

\section{Reproducing kernels and Fourier Integral}

We describe reproducing kernels in terms of canonical products. First, we define the set of divisors
$$
\cD(E)=\{D=\{(\l_j,\e_j)\}_{j\ge1}:\ \l_j\in[a_j,b_j],\ \e_j=\pm 1\}
$$
with the identification $(a_j,1)=(a_j,-1)$ and $(b_j,1)=(b_j,-1)$, endowed with the product topology of circles. To $D\in \cD(E)$ we associate
\begin{equation}\label{5oct1}
V(\l,D)=\left(
\prod_{j\ge1}
\sqrt{\frac{(\l_*-a_j)(\l_*-b_j)}{(\l-a_j)(\l-b_j)}}\frac{\l-\l_j}{\l_*-\l_j}\frac {1}{\Phi_{\l_j}(\l)}
\right)^{\frac 1 2}\prod_{j\ge 1}\Phi_{\l_j}(\l)^{\frac{1+\e_j}2}.
\end{equation}
We note that the product in the brackets represents an outer function, therefore the square root of this product is well defined and represents a character automorphic function. The second factor is an inner function (Blaschke product) in the given domain. Alternatively, we can write
\begin{equation}\label{5oct2}
V(\l,D)=\sqrt{O(\l,D) I(\l,D)}, \quad \text{where}\quad I(\l,D)=\prod_{j\ge 1}\Phi_{\l_j}^{\e_j}
\end{equation}
and
\begin{equation}\label{5oct3}
O(\l,D)=e^{\int \left(\frac 1{\xi-\l}-\frac 1{\xi-\l_*}\right)\chi_D(\xi)d\xi}, \quad 
\chi_D(\xi)=\begin{cases} 1/2,&\xi\in(a_j,\l_j)\\
-1/2,&\xi\in(a_j,\l_j)\\ 0, &\text{otherwise}
\end{cases}
\end{equation}

\begin{definition}
We define the Abel map $\cA:\cD(E)\to\pi_1(\O)^*$ by
$$
\cA(D)=\text{character of}\ V(\l,D).
$$
\end{definition}
\begin{theorem}[see \cite{SY97}]
For a Widom domain $\O$ if DCT holds, then the Abel map is a homeomorphism. 
\end{theorem}

In particular, $D=D(\a)$ is uniquely defined by 
$
\cA(D)=\a\in\pi_1(\O)^*.
$ To simplify notation  we write
$$
O_\a(\l):=O(\l,D),\ I_\a(\l):=I(\l,D),\ V_\a(\l):=V(\l,D)\quad \text{for}\quad \a=\cA(D).
$$

\begin{lemma} \label{l15oct}
Let $D_*:=\{(\l_j,-\e_j)\}_{j\ge 1}$. Then 
\begin{equation}\label{5oct4}
V(\l,D_*)=V_{\fj-\a(D)}(\l),\ \l\in\O,\quad V(\l,D_*)=\overline{V(\l,D)}, \ \l\in\sE.
\end{equation}
\end{lemma}

\begin{proof}
Note that $O(\l,D)$ assumes positive values on $\sE$, see \eqref{5oct3}, and $|I(\l,D)|=1$ here. Therefore
$$
\overline{V(\l,D)}V(\l,D)=O(\l,D), \quad \l\in\sE,
$$
and by \eqref{5oct2}
$$
\overline{V(\l,D)}=\frac{O(\l,D)}{V(\l,D)}=\frac{O(\l,D)}{\sqrt{O(\l,D) I(\l,D)}}=\sqrt{O(\l,D) I(\l,D)^{-1}}=V(\l,D_*).
$$
Further, $\sqrt{\l}O(\l,D)$ is single valued in the domain, hence $\fj$ is the character of $O(\l,D)$. In conjunction with the previous line this proves the first identity in 
\eqref{5oct4}. 
\end{proof}

\begin{lemma}\label{l18oct}
Let $\l_0\in\bbR_-$. By $\frac 1 2 \b_0$ we understand the character of the outer function $\sqrt{\Phi_{\l_0}(\l)/(\l-\l_0)}$. Then
\begin{equation}\label{5jan5}
k^\a_{\l_0}(\l)=\frac{\sqrt{\l_0}}{i}\frac{V_{\a-\frac 1 2\b_0}(\l)}{V_{\fj-\a+\frac 1 2\b_0}(\l_0)}\sqrt{\frac{\Phi_{\l_0}(\l)\Phi'_{\l_0}(\l_0)}{\l-\l_0}}.
\end{equation}
\end{lemma}

\begin{proof}
By DCT and Lemma \ref{l15oct} we have
\begin{equation*}
\langle F,k^\a_{\l_0}\rangle=\frac {\sqrt{\l_0}}{2\pi i}\oint_{\sE}\frac{V_{\fj-\a+\frac 1 2\b_0}(\xi)}{V_{\fj-\a+\frac 1 2\b_0}(\l_0)}\sqrt{\frac{\Phi'_{\l_0}(\l_0)}{(\xi-\l_0)\Phi_{\l_0}(\xi)}}
F(\xi)\frac{d\xi}{\sqrt{\xi}}=F(\l_0).
\end{equation*}
\end{proof}

A proof of Theorem \ref{thm1} is based essentially 
 on a  combination of Lemma \ref{l1oct10} and Theorem \ref{thm1bis} below.

\begin{lemma}\label{l1oct10} For a fixed $\l\in \bbC\setminus \bbR_+$ the function ${V_{\a}(\l)}$ is continuous on the compact Abelian group $\pi_1(\O)^*$. 
The following limit exists uniformly in $\a$.
$$
v_{\a,\l_*}(\a)=\lim_{\l_0\to-\infty}\frac{k^\a(\l,\l_0)}{k^{\a}(\l_*,\l_0)}=\frac{V_{\a}(\l)}{V_{\a}(\l_*)}.
$$
\end{lemma}

\begin{proof}
We note that both functions $O(\l,D)$ and $I(\l,D)$ are continuous in $D\in\cD(E)$ as soon as $\l\in\bbC\setminus \bbR_+$. By continuity of the Abel map we have continuity of $V_\a(\l)$.
Then, we pass to the limit using the representation \eqref{5jan5}.
Thus the first statement \eqref{pat8oct2} of Theorem \ref{thm1} is proved.
\end{proof}

\begin{theorem}\label{thm1bis}
For a Widom domain $\O$ with DCT we define a positive continuous measure $\vk^\a=\vk_{\l_*}^\a$ by \eqref{pat8oct3}.
Then
\begin{equation}\label{pa8oct1}
k^{\a}(\l,\l_0)-e^{2ix(\t(\l)-\overline{\t(\l_0)})}k^{\a-\eta x}(\l,\l_0)=\int_0^x
\overline{f_{\a,\l_*}(\l_0,\xi)}
f_{\a,\l_*}(\l,\xi)
d\vk^\a(\xi),
\end{equation}
where
$$
f_{\a,\l_*}(\l,x)=e^{i x(\t(\l)-\t(\l_*))}\frac{V_{\a-x\eta}(\l)}
{V_{\a-x\eta}(\l_*)}.
$$
In particular, 
\begin{equation}\label{pa8oct2}
k^{\a}(\l,\l_0)=\int_0^\infty
\overline{e^{i x(\t(\l_0)-\t_*)}v_{\a-\eta x,\l_*}(\l_0)}
e^{i x(\t(\l)-\t_*)}v_{\a-\eta x,\l_*}(\l)
d\vk^\a(x).
\end{equation}

\end{theorem}

\begin{proof} WLOG we set $x=1$.
First, we introduce a sequence of measures $\{\vk^\a_N\}_{N\ge 1}$ on $[0,1]$. By regularity for a fixed $N$ we define $\l_N<\l_*$ such that $\cG(\l_N,\l_*)=1/N$. 
We set
$$
\vk^\a_{N}\{ k/ N\}=|e_k^\a(\l_*,\l_N)|^2=|\Phi_{\l_N}(\l_*)|^k\frac{|k^{\a-k\b_N}(\l_*,\l_N)|^2}{k^{\a-k\b_N}(\l_*,\l_*)},\quad k=0,\dots, N-1.
$$
where $\b_N$ is the character generated by the function $\Phi_{\l_N}(\l)$
and $\vk_N^{\a}\{B\}$ is  the measure of a set $B$ (a single point in our case). By \eqref{4oct3} and \eqref{4oct4} we have
the standard for reproducing kernels relation
\begin{equation}\label{8oct1}
k^{\a}(\l,\l_0)=\sum_{n\ge 0}\overline{e^\a_n(\l_0,\l_N)}e^\a_n(\l,\l_N).
\end{equation}
Therefore, the corresponding distribution function can be simplified to the form
\begin{equation*}
\vk^\a_{N}(x)=\sum_{k=0}^{[Nx]}|e_k^\a(\l_*,\l_N)|^2=k^{\a}(\l_*,\l_*)-|\Phi_{\l_N}(\l_*)|^{2([Nx]+1)}k^{\a-([Nx]+1)\b_N}(\l_*,\l_*),
\end{equation*}
where $x\in (0,1)$ is irrational.

By \eqref{4oct5} and \eqref{84oct5} we have
$$
\Phi_{\l_N}(\l_*)^{[Nx]+1}=\left(e^{i\t_{\l_N}(\l_*)/\cG_{\l_N}(\l_*)}\right)^{\frac{[Nx]+1}{N}}\to e^{ix\t(\l_*)}, \quad\text{as}\ N\to\infty. 
$$
For the same reason, $e^{2\pi i N\beta_N(\g)}\to e^{2\pi i\eta(\g)}$ for all $\g\in\pi_1(\O)$. Using continuity of the reproducing kernel we obtain
\begin{equation}\label{8oct2}
\lim_{N\to\infty}\vk^\a_N(x)=k^{\a}(\l_*,\l_*)-e^{-2x\Im \t(\l_*)}k^{\a-\eta x}(\l_*,\l_*)=\vk^\a(x).
\end{equation}

Going back to the general expression \eqref{8oct1} for fixed $\l$ and $\l_0$ we write
\begin{align}\label{8oct4}
k^{\a}(\l,\l_0)-\Phi_{\l_N}(\l_*)^N\overline{\Phi_{\l_N}(\l_*)^N}k^{\a-N\b_N}(\l,\l_0)\\
=&\sum_{n=0}^{N-1}\overline{\frac{e^\a_n(\l_0,\l_N)}{e^\a_n(\l_*,\l_N)}}
|e^\a_n(\l_*,\l_N)|^2 \nonumber
\frac{e^\a_n(\l,\l_N)}{e^\a_n(\l_*,\l_N)}\\
=&\sum_{n=0}^{N-1}\overline{\frac{e^\a_n(\l_0,\l_N)}{e^\a_n(\l_*,\l_N)}}
\frac{e^\a_n(\l,\l_N)}{e^\a_n(\l_*,\l_N)}\vk^\a_{N}\{k/N\}.\nonumber
\end{align}
In the right hand side we can pass to the limit as it was discussed above
$$
k^{\a}(\l,\l_0)-\lim_{N\to\infty}\Phi_{\l_N}(\l)^N\overline{\Phi_{\l_N}(\l_0)^N}k^{\a-N\b_N}(\l,\l_0)
=k^{\a}(\l,\l_0)-e^{i(\t(\l)-\overline{\t(\l_0)})}k^{\a-\eta}(\l,\l_0).
$$
Due to Lemma \ref{l18oct} we have
$$
\frac{e^\a_k(\l,\l_N)}{e^\a_k(\l_*,\l_N)}=\sqrt{\frac{1-\l_*/\l_N}{1-\l/\l_N}}e^{i\frac {k+0.5} N(\t_{\l_N}(\l)-\t_{\l_N}(\l_*))/\cG_{\l_N}(\l_*)}\frac{V_{\a-\frac{k+0.5}{N}N\b_N}(\l)}
{V_{\a-\frac{k+0.5}{N}N\b_N}(\l_*)}.
$$
Therefore, for an arbitrary $\ve >0$,  for sufficiently big $N>N_0$
$$
\left| \frac{e^\a_k(\l,\l_N)}{e^\a_k(\l_*,\l_N)}-
e^{i\frac {k} N(\t(\l)-\t(\l_*))}\frac{V_{\a-\frac{k}{N}\eta}(\l)}
{V_{\a-\frac{k}{N}\eta}(\l_*)}
\right|\le \ve
$$
holds for all $k\le N$. Thus, the last expression in \eqref{8oct4} can be substituted with a fixed error by the integral
$$
\int_0^1
 e^{-i x(\overline{\t(\l_0)-\t(\l_*)})}\overline{\frac{V_{\a-x\eta}(\l_0)}
{V_{\a-x\eta}(\l_*)}}
e^{i x(\t(\l)-\t(\l_*))}\frac{V_{\a-x\eta}(\l)}
{V_{\a-x\eta}(\l_*)}d\vk_N^\a(x).
$$
Due to the measure convergence \eqref{8oct2} we obtain \eqref{pa8oct1}. Finally, we can pass to the limit as $x\to\infty$ in the left hand side of this relation,  we obtain \eqref{pa8oct2}.
\end{proof}

\begin{remark}
DCT, equivalently a continuity of the reproducing kernels with respect to the character, plays the key role in the proof of Theorem \ref{thm1}. Moreover, if it fails the chain of subspaces
$e^{ix\t}E^2_{\O}(\a-\eta x)$ is not necessary complete for a certain $\a\in\pi_1(\O)^*$,  as it was shown in \cite{Yl1}, what contradicts to the integral representation \eqref{pa8oct1}.
\end{remark}

\begin{proof}[Proof of Theorem \ref{thm1}]
On the dense set  in $E_{\O}^2(\a)$ we define a map to $\chi_{\bbR_+}L^2_{d\vk^\a_{\l_*}}$ by
\begin{equation}\label{9oct1}
e^{i x(\t(\l)-\t(\l_0))}k^{\a-\eta x}(\l,\l_0)\mapsto\chi_{(x,\infty)}(\xi)\overline{e^{i x(\t(\l_0)-\t(\l_*))}v_{\a-\eta \xi,\l_*}(\l_0)},
\end{equation}
where $\chi_{B}$ is the characteristic function of a set $B$.
By Theorem \ref{thm1bis} this is an isometry. Therefore this map is well defined on $E_{\O}^2(\a)$. Since in the image we have all  functions $\{\chi_{(x,\infty)}\}_{x\in\bbR_+}$, it is dense. Indeed, assume that $f(\xi)\in \chi_{\bbR_+}L^2_{d\vk^\a_{\l_*}}$ is orthogonal to this collection, then
$$
\int_{x}^\infty f(\xi)d\vk_{\l_*}^\a(\xi)=0.
$$
That is, $f(\xi)=0$ for a.e. $\xi$ with respect to the measure $\vk_{\l_*}^\a$. Thus $\cF^\a$ restricted to $\chi_{\bbR_+}L^2_{d\vk^\a_{\l_*}}$ is well defined as  the inverse to the map \eqref{9oct1}.

In fact, for the same reason, we have that $\cF^\a:\chi_{(x,\infty)}L^2_{d\vk^\a_{\l_*}}\to e^{ix\t }E_{\O}^2(\a-\eta x)$ acts unitary for all $x\in\bbR$. It remains to show that 
$$
\cup_{x\in\bbR}e^{ix\t }E_{\O}^2(\a-\eta x)=L^2_{\O}.
$$
Equivalently, by the property (iv) of Theorem \ref{thdct}, we have to show that
$$
\cap_{x\in\bbR}e^{-ix\t }E_{\O}^2(\fj-\a+\eta x)=\{0\}.
$$
If $F$ belongs to the intersection, then for every $x$ there exits $G_x\in E_{\O}^2(\fj-\a+\eta x)$ such that $F=e^{-ix\t }G_x$. 
In particular, $F\in  E_{\O}^2(\fj-\a)$.  So, it is enough to show that $F(\l)=0$ for all $\l\in\O$. 
Note that 
$\|G_x\|=\|F\|$. Since 
$$
|G_{x}(\l)|\le \| F\|\sqrt{k^{\fj-\a+\eta x}(\l,\l)}\quad\text{and}\quad C= \sup_{\b\in \pi_1(\O)^*}k^{\b}(\l,\l)<\infty,
$$
we have
$$
|F(\l)|\le C\|F\|e^{x\Im\t(\l)}\to 0, \quad\text{as}\quad x\to-\infty.
$$

\end{proof}

\begin{remark}
As a byproduct, we proved  that the  span of
$\{e^{i x(\t-\t_*)}k_{\l_*}^{\a-\eta x}\}$ is dense in $E_{\O}^2(\a)$ as $x\in\bbR_{+}$ and in $L^2_{\O}$ if the parameter $x$ runs in $\bbR$.
\end{remark}

\section{Reproducing kernels and Transfer matrices}

In plain domains the reproducing kernels of Hardy/Smirnov spaces have a very specific structure: a kind of resolvent expression related to the operator multiplication by the independent variable $\l$.

\begin{proposition}
The reproducing kernel of $E^2_{\O}(\a)$ is of the form
\begin{equation}\label{b12mar5}
k^\a(\l,\l_0)
=i\cC(\a)\frac{\sqrt{\l}V_{\a+\fj}(\l) \overline{V_{\a}(\l_0)}+V_{\a}(\l) \overline{\sqrt{\l_0}V_{\a+\fj}(\l_0)}}{\l-\bar \l_0},
\end{equation}
where $\cC(\a)=\cC_{\l_*}(\a)$ is given by
\begin{equation}\label{b12apr7}
\frac 1{\cC_{\l_*}(\a)}=
V_{-\a}(\l_*) {V_{\a}(\l_*)}+V_{\fj-\a}(\l_*) V_{\a+\fj}(\l_*)=\sqrt{\frac{I_\a(\l_*)}{I_{\a+\fj}(\l_*)}}+\sqrt{\frac{I_{\a+\fj}(\l_*)}{I_{\a}(\l_*)}}.
\end{equation}
\end{proposition}

\begin{proof}
Note that $V_\a(\l)/(\l-\l_0)\in L^2_\O$, $\l_0\in\O$, and even after multiplication by $\sqrt{\l}$ we still have a function from $L^2_{\O}$. Since
$$
G(\l)=i\frac{\sqrt{\l}V_{\a+\fj}(\l) {V_{\a}(\bar \l_0)}-V_{\a}(\l) {\sqrt{\bar \l_0}V_{\a+\fj}(\bar\l_0)}}{\l-\bar \l_0}
$$
is of Smirnov class $N_+(\O)$, we get that $G(\l)\in E^2_{\O}(\a)$.
Note that by the construction $V_{\a}$ is real on $\bbR_-$ and $\sqrt{\l}$ takes pure imaginary values  here, so we can rewrite $G(\l)$ into the form
\begin{equation}\label{b12oct10}
G(\l)=i\frac{\sqrt{\l}V_{\a+\fj}(\l) \overline{V_{\a}(\l_0)}+V_{\a}(\l) \overline{\sqrt{\l_0}V_{\a+\fj}(\l_0)}}{\l-\bar \l_0}.
\end{equation}
We use Lemma \ref{l15oct} and DCT, then for an arbitrary $F(\l)\in E_{\O}^2(\a)$ we have
\begin{align}\label{10oct1}
\langle F,G \rangle=&\frac 1{2\pi i } \oint_{\sE}
\frac{\sqrt{\xi}V_{-\a}(\xi) {V_{\a}(\l_0)}+V_{\fj-\a}(\xi) {\sqrt{\l_0}V_{\a+\fj}(\l_0)}}{\xi- \l_0} F(\xi)\frac{d\xi}{\sqrt{\xi}}
\nonumber
\\
=&(V_{-\a}(\l_0) {V_{\a}(\l_0)}+V_{\fj-\a}(\l_0) V_{\a+\fj}(\l_0))F(\l_0).
\end{align}
In particular,
$$
\|G\|^2=(V_{-\a}(\l_0) {V_{\a}(\l_0)}+V_{\fj-\a}(\l_0) V_{\a+\fj}(\l_0))G(\l_0).
$$
Since $G$ is not identically zero $\|G\|^2 >0$. Since $G(\l_0)$ is a real number, we obtain that the analytic function 
$V_{-\a}(\l_0) {V_{\a}(\l_0)}+V_{\fj-\a}(\l_0) V_{\a+\fj}(\l_0)$ assumes only real values for all $\l_0\in\O$. Therefore it is constant, which we denote by
$1/\cC_{\l_*}(\a)>0$.  We have
\eqref{b12apr7}, see \eqref{5oct2} and \eqref{5oct3}. Consequently, \eqref{b12oct10} and \eqref{10oct1} implies \eqref{b12mar5}.

\end{proof}

\begin{definition}
In what follows, the relation 
\begin{equation}\label{10oct6}
\cC(\a)\det
\begin{bmatrix} V_{\a+\fj}(\l)&V_{\a}(\l)\\
-V_{-\a}(\l)& V_{\fj-\a}(\l)
\end{bmatrix}=1
\end{equation}
we call the Wronskian identity.
\end{definition}

\begin{corollary}
The generalized eigenfunction $v_{\a,\l_*}(\l)$ possesses the following representation in terms of reproducing kernels
\begin{equation}\label{18apr7}
\frac{i v_{\a,\l_*}(\l)}{I_\a(\l_*)+I_{\a+\fj}(\l_*)}=
i\cC(\a) V_{\a+\fj}(\l_*)V_{\a}(\l)=\sqrt{\l}k^{\a+\fj}(\l,\l_*)+\sqrt{\l_*}k^{\a}(\l,\l_*).
\end{equation}
\end{corollary}

\begin{proof}
Note that by \eqref{b12apr7} $\cC_{\l_*}(\a)=\cC_{\l_*}(\a+\fj)$. By \eqref{b12mar5} we have
$$
\begin{bmatrix} k^{\a}(\l,\l_*)\\k^{\a+\fj}(\l,\l_*)
\end{bmatrix}=i\cC_{\l_*}(\a)
\begin{bmatrix}\sqrt{\l}&\sqrt{\l_*}\\ \sqrt{\l_*}&\sqrt{\l}
\end{bmatrix}^{-1}
\begin{bmatrix}
V_{\a+\fj}(\l)V_{\a}(\l_*)\\
V_{\a}(\l)V_{\a+\fj}(\l_*)
\end{bmatrix}.
$$
Due to Lemma \ref{l1oct10}, we have \eqref{18apr7}.

\end{proof}

Further,
we will use a certain very general construction related to  the theory of extensions of isometries \cite{AAK2, AG},
either the functional models of contractive operators \cite{NF}, 
or the Lax-Phillips scattering theory, or the generalized interpolation in the sense of Potapov's approach \cite{KY} and so on...

Consider the Cayley transformation of the multiplication by $\l$ in $L^2_\O$,
$$
\u(\l)=\frac{\l-\l_0}{\l-\overline{\l_0}}=\frac{\Phi_{\l_0}(\l)}{\Phi_{\overline{\l_0}}(\l)}, \quad \Im\l_0>0.
$$
Note that both complex Green functions have the same character $\b_0$. Evidently
$$
\overline{\u}: \Phi_{\l_0} E^2_{\O}(\a-\b_0)\to  \Phi_{\overline{\l_0}} E^2_{\O}(\a-\b_0)
$$
acts unitary. 
From this relation, passing to orthogonal complements we obtain that the multiplication by
$\overline{\u(\l)}$ acts unitary from
\begin{equation}\label{10oct4}
\left\{\frac{e_{\overline{\l_0}}^{\a+\b_0}}{\Phi_{\overline{\l_0}}}\right\}\oplus \left(E^2_\O(\a)\ominus e^{ix\t}E_\O^2(\a-\eta x)\right)\oplus e^{ix\t}e_{\l_0}^{\a-\eta x}
\end{equation}
to 
\begin{equation}\label{10oct5}
\left\{\frac{e_{{\l_0}}^{\a+\b_0}}{\Phi_{{\l_0}}}\right\}\oplus \left(E^2_\O(\a)\ominus e^{ix\t}E_\O^2(\a-\eta x)\right)\oplus e^{ix\t}e_{\overline{\l_0}}^{\a-\eta x},
\end{equation}
where
\begin{equation*}
e^\a_{\l_0}=\frac{k_{\l_0}^{\a}}{\|k^\a_{\l_0}\|}.
\end{equation*}

Now we recall the notion of the \textit{unitary node} \cite{AG, KY}. This is a unitary operator $\cU$ acting from the space
$\cK\oplus \cE_1$ to $\cK\oplus \cE_2$. $\cK$ is called the internal space and $\cE_{1,2}$ are called the scaling spaces.
Evidently, in the decompositions \eqref{10oct4} and \eqref{10oct5} we have the unitary node with two dimensional scaling spaces and the internal space $\cK(\a,x)=E^2_\O(\a)\ominus e^{ix\t}E_\O^2(\a-\eta x)$.

The \textit{scattering matrix} of the unitary node is a contractive operator valued  analytic function $S(\z)$, $\z\in\bbD$, acting from $\cE_1$ to $\cE_2$ 
(for a fixed $\z$) and given by
$$
S(\z)=P_{\cE_2}(I-\z P_{\cK}\cU)^{-1}\cU|_{\cE_1},
$$
where $P_{\cK}$, $P_{\cE_2}$ are the orthogonal projections on the corresponding spaces. Note that as soon as we fix basises in the scaling spaces we get a  scattering matrix valued analytic function.

Applying this construction in our case $\cU=\overline{\u(\l)}$, we obtain the scattering matrix
\begin{equation}\label{9oct}
\begin{bmatrix}\frac{e_{{\l_0}}^{\a+\b_0}}{\Phi_{{\l_0}}}
 &e^{ix\t}e_{\overline{\l_0}}^{\a-\eta x}
\end{bmatrix}
S(\z)\begin{bmatrix} c_1\\ c_2
\end{bmatrix}=P_{\cE_2}(I-\z P_{\cK(\a,x)}\cU)^{-1}\cU
\begin{bmatrix}\frac{e_{\overline{\l_0}}^{\a+\b_0}}{\Phi_{\overline{\l_0}}}
 &e^{ix\t} e_{{\l_0}}^{\a-\eta x}
\end{bmatrix}\begin{bmatrix} c_1\\ c_2
\end{bmatrix}.
\end{equation}

Switching of  ``channels" to a more  natural pairs
$$
\begin{bmatrix}\frac{e_{{\l_0}}^{\a+\b_0}}{\Phi_{{\l_0}}}
 &\frac{e_{\overline{\l_0}}^{\a+\b_0}}{\Phi_{\overline{\l_0}}}
\end{bmatrix}\quad \text{and}\quad
e^{ix\t}\begin{bmatrix}e_{{\l_0}}^{\a-\eta x}
 &e_{\overline{\l_0}}^{\a-\eta x}
\end{bmatrix}
$$
leads to the so-called Potapov-Ginzburg transform of $S(\z)$:
\begin{equation}\label{910oct}
A(\z)=\begin{bmatrix} S_{11}(\z)&S_{12}(\z) \\
0 & 1
\end{bmatrix}
\begin{bmatrix} 1&0\\S_{21}(\z)&S_{22}(\z) 
\end{bmatrix}^{-1}.
\end{equation}
$A(\z)$ is called the transfer matrix and possesses a much easier chain property with respect to $x$.
The contractive property of the scattering matrix, $I-S(\z)S(\z)^*\ge 0$, perturbs to the $j$-contractive property of the transfer matrix, 
\begin{equation}\label{911oct}
j-A(\z)jA(\z)^*\ge 0,\quad  j=\begin{bmatrix}1&0\\ 0& -1
\end{bmatrix},
\end{equation}
see e.g. the beginning of  Section 6 in \cite{KY}.


\begin{theorem}
Let
\begin{equation}\label{27apr5}
\cV_\a(\l)=\sqrt{\cC_{\l_*}(\a)}\begin{bmatrix}i\sqrt\l V_{\a+\fj}(\l)&
V_{\a}(\l) \\
-i{\sqrt\l}V_{-\a}(\l)&
 V_{\fj-\a}(\l) 
\end{bmatrix},\quad \cJ=\begin{bmatrix}0&1\\ -1&0\end{bmatrix}.
\end{equation}
Then
 the family of transfer matrices $\cA_{\a}(\l,x)$ is given by
\begin{equation}\label{27apr4}
\begin{bmatrix} e^{ix\t(\l)}&0\\ 0&e^{-ix\t(\l)}
\end{bmatrix}
\cV_{\a-\eta x}(\l)=\cV_{\a}(\l)\cA_\a(\l,x).
\end{equation}
They form a monotonic family of $\cJ$-contractive entire matrix functions.
\end{theorem}

\begin{proof}
Formulas 
\eqref{27apr4} and \eqref{27apr5} follow from 
\eqref{9oct} and
\eqref{910oct} as soon as we take into account the representation \eqref{b12mar5} for reproducing kernels.
We performed these computations several times, see for details e.g. \cite[Appendix]{VYKL}. 
Note that the modification of the basis functions in the scaling spaces leeds to the another form of $\cJ$-matrix.
Respectively, \eqref{911oct} has the form
$$
\frac{\cJ-\cA_\a(\l,x)\cJ\cA_\a(\l,x)^*}{\l-\bar\l}\ge 0.
$$
Particularly, for the upper corner entry  we have here
$$
\left\{\cV_{\a}(\l)\frac{\cJ-\cA_\a(\l,x)\cJ\cA_\a(\l_0,x)^*}{\l-\bar\l_0}\cV_{\a}(\l_0)^*\right\}_{11}
=k^{\a}(\l,\l_0)-e^{i x(\t(\l)-\overline{\t(\l_0)})}k^{\a-\eta x}(\l,\l_0).
$$
The chain property 
\begin{equation}\label{16oct7}
\cA_\a(\l,x_1+x_2)=\cA_\a(\l,x_1)\cA_{\a-\eta x_1}(\l,x_2)
\end{equation}
follows immediately from the representation \eqref{27apr4}.

The fact that $\cA_{\a}(\l,x)$ is an entire matrix function  requires again the DCT property. We use the following lemma, which we prove later on.
\begin{lemma}\label{ledct}
 If $\O=\bbC\setminus E$ is of Widom type and DCT holds, then  $\O_n:=\O\cap\bbD_{\frac{a_n+b_n}2}$ is of Widom type and DCT holds in it.
\end{lemma}

Due to the Wronskian identity $\cA_\a(\l,x)$ is holomorphic in $\O$. We have to consider the boundary points $\l=\xi\pm i0$, $\xi\in \sE$.
Note that for such $\l$
$$
\overline{\cV_{\a}(\l)}=\sqrt{\cC_{\l_*}(\a)}\begin{bmatrix}-i\sqrt\l V_{-\a}(\l)&
V_{\fj-\a}(\l) \\
i{\sqrt\l}V_{\a+\fj}(\l)&
 V_{\a}(\l) 
\end{bmatrix}=\begin{bmatrix}0&1\\1&0
\end{bmatrix}\cV_{\a}(\l)
$$
and $\overline{e^{ix\t(\l)}}=e^{-ix\t(\l)}$. Therefore,  $\overline{\cA_{\a}(x,\l)}=\cA_{\a}(x,\l)$ and boundary values at $\l=\xi\pm i0$ coincides. 
Let us write explicitly the entries of the transfer matrix,
$$
\frac{\cA_{\a}(\l,x)}{\sqrt{\cC(\a)\cC(\a-\eta x)}}=\begin{bmatrix}\frac{V_{\fj-\a}(\l)}{i\sqrt{\l}}  &
-\frac {V_{\a}(\l)}{i\sqrt{\l}} \\
V_{-\a}(\l)&
 V_{\a+\fj}(\l)
\end{bmatrix}
\begin{bmatrix}i\sqrt\l  e^{ix\t(\l)}V_{\a+\fj-\eta x}(\l)&
 e^{ix\t(\l)}V_{\a-\eta x}(\l) \\
-i{\sqrt\l} e^{ix\t(\l)}V_{\eta x-\a}(\l)&
  e^{ix\t(\l)}V_{\fj+\eta x-\a}(\l) 
\end{bmatrix}.
$$
It is easy to see that the entries of $\sqrt{\l}\cA_{\a}(\l,x)$ belongs to $E^1_{\O_n}(\hat\fj)$,
 $\hat \fj=\fj |\pi_1( \O_n)$. Applying DCT in this domain we have
$$
\cA_{\a}(\l, x)=\frac 1{2\pi i}\oint_{\pd\hat\O_n}\frac{\sqrt{\xi}}{\xi-\l}\cA_{\a}(\xi, x)\frac{d\xi}{\sqrt{\xi}}=
\frac 1{2\pi i}\oint_{2|\xi |=a_n+b_n}\cA_{\a}(\xi, x)\frac{d\xi}{\xi-\l}.
$$
That is, $\cA_{\a}(\l, x)$ is holomorphic at an arbitrary disk $\bbD_{\frac{a_n+b_n}2}$.
\end{proof}

\begin{proof}[Proof of Lemma \ref{ledct}] The proof is based on the property (iii) in Theorem \ref{thdct}.
It is evident that $H_{\O}^\infty(\a)\subset H_{\O_n}^\infty(\hat\a)$, $\hat \a=\a| \pi_1( \O_n)$.
Further, let $\hat\a_m\to 1_{\pi^*_1(\O_n)}$. We define $\a_m$ such that for all generators $\g_j$ of $\pi_1(\O)$
$$
\a_m(\g_j)=\hat \a_m(\g_j)\ \text{if}\ \g_j\in  \pi_1(\O_n)\ \text{and}\ \a_m(\g_j)=1\ \text{otherwise}.
$$
Evidently $\a_m\to 1_{\pi^*_1(\O)}$.  For the minimizer $\hat \cW_{\hat \a_m}\in H^\infty_{\O_m}$ we have
$$
\hat \cW_{\hat \a_m}(\l_*)\ge \cW_{\a_m}(\l_*)\to 1\quad\text{as}\ m\to \infty.
$$
\end{proof}

\begin{remark}
If DCT fails, then remaining singularities on $\sE$ for a transfer matrix of the form \eqref{27apr4} are possible \cite{Yl1}.
\end{remark}

\section{Weyl-Titchmarsh functions}
Matrices
\begin{equation}\label{11oct1}
\cA_{\a}(\l,x)^{-1}=\begin{bmatrix} A_\a(\l,x)&C_{\a}(\l,x)\\
B_{\a}(\l.x)&D_{\a}(\l,x),
\end{bmatrix}
\end{equation}
which were defined in the previous section, form a monotonic family of $\cJ$ expanding matrix functions in $\bbC_+$. For an arbitrary such family  the Weyl circle is formed by values
\begin{equation*}
m=\frac{A(\l,x)w+B(\l,x)}{C(\l,x)w+D(\l,x)},\quad w\in\bbR\cup\{\infty\},
\end{equation*}
where  $\l\in\bbC_+$ and $x\in\bbR_+$ are fixed.
These circles are nesting in the upper half plane as $x$ increases, and in the limit converge either to a circle  or to a point. Due to the explicit formula \eqref{27apr5} we have the limit point case.

\begin{theorem}\label{th11oct1}
For an arbitrary Nevanlinna class function $w(\l)$ the following limit exists
\begin{equation}\label{11oct2}
m^\a_+(\l):=\lim_{x\to \infty}\frac{A_\a(\l,x)w(\l)+B_\a(\l,x)}{C_\a(\l,x)w(\l)+D_{\a}(\l,x)}=i\sqrt{\l}\frac{V_{\a+\fj}(\l)}{V_{\a}(\l)}.
\end{equation}
Moreover,
\begin{equation}\label{11oct3}
m^\a_-(\l)=-\overline{m^\a_+(\l)}=i\sqrt{\l}\frac{V_{-\a}(\l)}{V_{\fj-\a}(\l)}, \quad \text{for a.e.}\ \l\in\sE,
\end{equation}
and
\begin{equation}\label{11oct4}
R^\a(\l):=-\frac{1}{m_+^\a(\l)+m_-^\a(\l)}=\frac{i\cC_{\l_*}(\a)}{\sqrt\l}\prod_{j\ge 1}\frac{(\l-\l_j)\sqrt{(\l_*-a_j)(\l_*-b_j)}}{(\l_*-\l_j)\sqrt{(\l-a_j)(\l-b_j)}},
\end{equation}
where $\a=\cA(D)$, $D\in\cD(E)$.
\end{theorem}

Note that $m_{\pm}$ belongs to the Stieltjes, we will call them Weyl-Titchmarsh  functions. The property \eqref{11oct3} is  reflectionless (on the set $\sE$). 

\begin{proof}[Proof of Theorem \ref{11oct1}]
\eqref{11oct2} follows directly from the definition \eqref{27apr5}. \eqref{11oct3} follows from Lemma \ref{l15oct}. \eqref{11oct4} is a consequence of the Wronskian identity and the definition of $V_{\a(D)}$.

\end{proof}
In a sense we will invert Theorem \ref{th11oct1}.

\begin{definition} Let $\sE=\bbR_+\setminus\cup_{j\ge 1}(a_j,b_j)$ be  of positive Lebesgue measure.
We say that $m_+\in\cS$ belongs to the set $m(\sE)$ if this function is reflectionless on $\sE$, that is, there exists $m_-\in \cS$ such that 
$m_-(\l)=-\overline{m_+(\l)}$ for a.e. $\l\in\sE$, and the both functions
\begin{equation}\label{11oct6}
R_0(\l)=-\frac{1}{m_+(\l)+m_-(\l)},\quad R_1(\l)=\frac{m_+(\l)m_-(\l)}{m_+(\l)+m_-(\l)}
\end{equation}
are holomorphic in $\O=\bbC\setminus \sE$, being extended in the lower half plane due to the symmetry principle
$\overline{R _i(\bar\l)}=R_i(\l)$, $\l\in \O$.
\end{definition}

Note that a function of the Nevanlinna class  is defined uniquely by its limit values on a set of positive Lebesgue measure. Thus $m_+$ defines 
$m_-$  uniquely.

 Further, automatically, $R_i$ belongs to the Nevanlinna class. 
Therefore they can be restored (up to positive multipliers) by their arguments on the real axis due to \eqref{11oct8}.
The boundary values of $R_i$ on $E$ are pure imaginary, that is, $\arg R_i(\xi)=\pi/2$, a.e. $\xi\in\sE$. On the complement $\bbR\setminus \sE$ they are real.
Since $R_i$ is increasing in each gap, there is a unique point $\l_j^{(i)}\in[a_j,b_j] $, $j\ge 0$, such that $\arg R_{i}(\xi)=0$ in $(\l^{(i)}_j, b_j)$ (respectively, 
$\arg R_{i}(\xi)=\pi$ in 
 $(a_j, \l_j)$; one of these sets could be empty). 
 As the result we have
\begin{equation}\label{11oct10}
 R_i(\l)=R_i(\l_*)\frac{\l-\l^{(i)}_0}{\l_*-\l^{(i)}_0}\sqrt{\frac{\l_*}{\l}}\prod_{j\ge 1}\frac{(\l-\l^{(i)}_j)\sqrt{(\l_*-a_j)(\l_*-b_j)}}{(\l_*-\l^{(i)}_j)\sqrt{(\l-a_j)(\l-b_j)}},
\end{equation}
 where $\l_*<0$ is a normalization point, $\l_*\not=\l^{(i)}_0$, $-\infty\le \l_0^{(i)}\le 0$. That is, $R_i(\l)$ is completely defined by the collections of $\{\l^{(i)}_j\}_{j\ge 0}$ and $R_{i}(\l_*)$.
 
 \begin{definition}
If  $m_+\in m(\sE)$ meets the additional conditions 
 \begin{itemize}
 \item[(a)] $\l_0^{(0)}=-\infty$ and $\l_0^{(1)}=0$,
 \item[(b)] along the negative half axis $ \lim_{\l\to -0} m_+(\l)=0,
 $
 \end{itemize}
 we say that $m_+\in m_0(\sE)$.
 \end{definition}
 
 Note that if (a) holds, then the increasing function $m_+$ is bounded on the positive half axis, that is, the limit exists, but not necessarily $0$. Thus (b) is a certain (additive) normanlization condition.
 
 Going back to Widom domains with DCT we have the following  important property.
 \begin{theorem}[see \cite{PoRem, Yud11}]
Let $\O=\bbC\setminus\sE$ be of Widom type and DCT hold. Assume that $R_i(\l)$ is of the form \eqref{11oct10} corresponding to an arbitrary collection of $\l_j^{(i)}\in [a_j,b_j]$, 
$-\infty \le \l^{(i)}_0\le 0$. Then  the measures, corresponding to the Nevanlinna functions $\pm R_i(\l)^{\pm 1}$  in their integral representations \eqref{11oct7}, are absolutely continuous on $\sE$. In particular, this implies that for $\l_0^{(0)}=-\infty$
\begin{equation}\label{11oct11}
\lim_{\l\to -0}R_0(\l)=\infty,\quad \lim_{\l\to -\infty}R_0(\l)=0.
\end{equation}
\end{theorem}

\begin{corollary}
If $\O$ is of Widom type and DCT holds then $m_+\in m_0(\sE)$ implies $m_{-}\in m_0(\sE)$.
\end{corollary}
 \begin{proof}
 By \eqref{11oct11}.
 \end{proof}
 
 \begin{remark}
 Once again we note importance of DCT property. A singular component for a measure, associated to a reflectionless function,  is possible if  DCT fails in a Widom domain, as well as $m_-$ not necessarily belongs to $m_0(\sE)$ for
 a certain $m_+\in m_0(E)$.
 \end{remark}
 
 \begin{theorem} \label{thoct12}
 As soon as DCT holds, 
 one can parametrize the set $m_0(\sE)$ by the following collection of data $\{R_0(\l_*),D\}\in\bbR_+\times \cD(\sE)$.
 \end{theorem}
 
 \begin{proof}
 First, we recall that the Nevanlinna functions $w$, which can be extended by the symmetry through $\bbR_-$ in the lower half plane  and such that
 $\lim_{\l\to-0}w(\l)=0$, allows the following representation
\begin{equation}\label{12oct1}
w(\l)=a\l+\int_{\bbR_+}\frac{\l}{\xi-\l}d\s(\xi),\quad a\ge 0,\  \int_{\bbR_+}\frac{d\sigma(\xi)}{1+\xi}<\infty.
\end{equation}
Indeed, we can represent it in the form
$$
w(\l)=w(\l_0)+a(\l-\l_0)+\int_{\bbR_+}\left(\frac{1}{\xi-\l}-\frac{1}{\xi-\l_0}\right)d\tilde \s(\xi)\quad a\ge 0,\  \int_{\bbR_+}\frac{d\tilde\sigma(\xi)}{1+\xi^2}<\infty,
$$
where $\l_0<0$. Then, pass to the limit as $\l_0\to-0$. We get \eqref{12oct1} with $d\s=\frac 1 {\xi}d \tilde\s \ge 0$.

Let now $m_+\in m_0(\sE)$. Then we have the collection $R_0(\l_*)$ and $\l_j\in[a_j,b_j]$ such that
\begin{equation}\label{12oct2}
 R_0(\l)=R_0(\l_*)\sqrt{\frac{\l_*}{\l}}\prod_{j\ge 1}\frac{(\l-\l_j)\sqrt{(\l_*-a_j)(\l_*-b_j)}}{(\l_*-\l_j)\sqrt{(\l-a_j)(\l-b_j)}}.
\end{equation}
On the other hand, by \eqref{12oct1}
$$
-\frac 1{R_{0}(\l)}= a\l+\int_{\sE}\frac{\l}{\xi-\l}d\s_0(\xi)+\sum_{\l_j\in(a_j,b_j)}\frac{\l\s^{(0)}_j}{\l_j-\l}.
$$
Since, there is no mass points on $\sE$, including infinity, we have $a=0$. The measure $d\sigma_0$ is absolutely continuous. Due to
$$
-\frac 1{R_{0}(\l)}=m_+(\l)+m_-(\l)
$$
we have to distribute this measure as $d\s_0=d\s_++d\sigma_-$. That is, both measures are absolutely continuous and due to
 $\Im m_+(\xi)=\Im m_{-}(\xi)$, $\xi\in\sE$, they are equal. With respect to $\s_j^{(0)}$, if $\s_j^{(0)}=\s_j^{+}+\s_j^{-}$ and both values are positive, then
 $R_1(\l)$ has a pole in $\l_j$, see \eqref{11oct6}. We write $\e_j=\pm 1$ if  $\s_j^{(0)}=\s_j^{\pm}$.
 
 As the  result we have
 \begin{equation}\label{12oct3}
m_{\pm}(\l)=\frac 1 2\left(-\frac 1{R_{0}(\l)}
\pm \sum_{\l_j\in(a_j,b_j)}\frac{\l\s^{(0)}_j\e_j}{\l_j-\l}
\right).
\end{equation}
Vice versa, for an arbitrary collection from $\bbR_+\times \cD(\sE)$, we define $R_0(\l)$ by \eqref{12oct2} and
$m_+(\l)$ by \eqref{12oct3}. We have $m_+(\l)\in m_0(\sE)$.
 \end{proof}  
 
 \begin{remark}
 Comparing \eqref{11oct2} and \eqref{12oct3} we get that, in the sense of Theorem \ref{thoct12}, to the function $m_+^\a$ corresponds exactly that divisor $D$ for which $\cA(D)=\a$.
 Comparing \eqref{11oct4} and \eqref{12oct2} we have $\sqrt{\l_*} R_0(\l_*)=i\cC_{\l_*}(\a)$.
 \end{remark}
 
 \begin{lemma}
 For a function $m_+\in m_0(\sE)$ there exits a unique representation
 \begin{equation}\label{120ct5}
m_+(\l)=i\sqrt{\l}\frac{V_2(\l)}{V_{1}(\l)},
\end{equation}
where $V_1,V_2\in\cN_+(\O)$
with mutually simple inner parts, which obey the Wronskian identity
$$
\sqrt{|\l_*|}R_0(\l_*)\det\begin{bmatrix} V_2(\l) & V_1(\l)\\
-\overline{V_2(\l)}& \overline{V_1(\l)}
\end{bmatrix}=1, \quad \l\in\sE.
$$
 \end{lemma}
 \begin{proof}
 We use essentially Theorem D \cite{SY97}, according to which $m_+$ is of bounded characteristic in $\O$ and has no singular component in its inner part.
For $\l=\xi\pm i0$, $\xi\in\sE$, we have
 $$
 \frac {m_+(\l)-\overline{m_+(\l)}} i=\frac {m_+(\l)+{m_-(\l)}} i=\frac i{R_0(\l)}.
 $$
 Having in mind the Wronskian identity, we get
 $$
 \frac i{ R_0(\l_*)}\sqrt{\frac{{\l}}{\l_*}}\frac 1{|V_1(\l)|^2}=\frac i{ R_0(\l_*)}\sqrt{\frac{{\l}}{\l_*}}\prod_{j\ge 1}\frac{(\l_*-\l_j)\sqrt{(\l-a_j)(\l-b_j)}}{(\l-\l_j)\sqrt{(\l_*-a_j)(\l_*-b_j)}}.
 $$ 
 That is,
 $
 |V_1|^2=O(\l,D),
 $
 which define uniquely the outer part of $V_1$. By \eqref{120ct5} its inner part is the Blaschke product $\prod_{j\ge 1}\Phi_{\l_j}^{\frac{1+\e_j} 2}$.
 Thus, $V_1(\l)=V(\l,D)$, see the definition \eqref{5oct1}, and $i\sqrt\l V_2(\l)=m_+(\l) V(\l,D)$. Moreover, since on $\sE$
 $$
 m_+(\l)m_{-}(\l)=-|m_+(\l)|^2=\frac{R_1(\l)}{R_0(\l)},
 $$
 we have
 $$
 \l|V_{2}(\l)|^2=\frac{m_+(\l_*) m_-(\l_*)}{-\l_*}{\l} \prod_{j\ge 1}\frac{\l-\l_j^{(1)}}{\l_*-\l_j^{(1)}}
 \frac{\l_*-\l_j}{\l-\l_j} |V_1(\l)|^2.
 $$
 We define $\e_j^{(1)}$ such that  $\prod_{j\ge 1}\Phi_{\l_j^{(1)}}^{\frac{1+\e^{(1)}_j} 2}$ is  the numerator of the inner part of $m_+(\l)$, then
 $$
 V_2(\l)=\sqrt{\frac{m_+(\l_*)m_-(\l_*)}{-\l_*}}V(\l,D^{(1)})\quad
 \text{and}\quad \cA(D^{(1)})=\a+\fj.
 $$
 \end{proof}
\section{Canonical systems}

In this section we will get consequences of the Fourier representation, Theorem \ref{thm1}.

\begin{lemma}[Main lemma] For all $x$ the we can pass to the limits  as $\l\to -0$ in the ratio $v_{\a-\eta x}(\l)/v_{\a}(\l)$,
see \eqref{15oct10}, \eqref{28sep1t}. Moreover, the following limit exists  as $\l$ approach infinity along the negative half axis
\begin{equation}\label{16oct1}
\lim_{\l\to-\infty}\frac 1 {\t(\l)}\int_0^x\fm_+^{\a-\eta \xi}(\l)\frac{e^{2\xi\Im\t_*}d\vk^{\a+\fj}(\xi)}{\fc(\a+\fj-\eta\xi)}=x.
\end{equation}

\end{lemma}

\begin{proof} Let 
\begin{equation}\label{13oct0}
\fc(\a)= \fc(\a+\fj)=V_{\a}(\l_*)V_{\a+\fj}(\l_*)\cC(\a)=\frac{\sqrt{\l_*}} i\left(k^{\a}(\l_*,\l_*)+k^{\a+\fj}(\l_*,\l_*)\right).
\end{equation}
Using \eqref{18apr7}, from \eqref{pa8oct1} we have  an integral relation 
\begin{align}\label{b13apr1}
\nonumber i \fc(\a-\eta x)v_{\a-\eta x}(\l)e^{ix(\t(\l)-\bar \t_*)}
=\sqrt{\l}\int_{x}^\infty e^{i\xi(\t(\l)-\t_*)} {v_{\a+\fj-\eta\xi}(\l)}  d\vk^{\a+\fj}(\xi)\\+
\sqrt{\l_*}\int_{x}^\infty  e^{i\xi(\t(\l)-\t_*)} {v_{\a-\eta\xi}(\l)}  d\vk^{\a}(\xi).
\end{align}
In terms of differentials we get
\begin{align*}
-i d\log \left( \fc(\a-\eta x) v_{\a-\eta x}(\l)e^{ix(\t(\l)-\bar \t_*)}\right)=\sqrt{\l}\frac{v_{\a+\fj-\eta x}(\l)}{ v_{\a-\eta x}(\l)}
\frac{e^{2x\Im \t_*}d\vk^{\a+\fj}(x)}{ \fc(\a+\fj-\eta x) }\\ +
\sqrt{\l_*}\frac{e^{2x\Im\t_*}d\vk^\a(x)}{\fc(\a-\eta x) }.
\end{align*}
Integrating on the interval $(0,\ell)$, we obtain
\begin{align}\label{13oct1}
({\t(\l)-\bar \t_*})\ell-i\log \frac{ \fc(\a-\eta \ell)v_{\a-\eta \ell}(\l)}{ \fc(\a)v_{\a}(\l)}\nonumber\\ =
\int_0^\ell
\frac{\sqrt{\l} v_{\a+\fj-\eta x}(\l)}{v_{\a-\eta x}(\l)}
\frac{e^{2\xi\Im\t_*}d\vk^{\a+\fj}(x)}{\fc(\a+\fj-\eta x) }
+\int_0^\ell \sqrt{\l_*}\frac{e^{2\xi\Im\t_*}d\vk^\a(x)}{ \fc(\a-\eta x) }.
\end{align}
For $\l=\l_*$ we have
\begin{align}\label{13oct2}
2\ell\Im\t_*-\log \frac{ \fc(\a-\eta \ell)}{ \fc(\a)}\nonumber\\ =
\frac{\sqrt{\l_*}} i \int_0^\ell
\frac{e^{2\xi\Im\t_*}d\vk^{\a+\fj}(x)}{ \fc(\a+\fj-\eta x) }
+\frac{\sqrt{\l_*}} i \int_0^\ell \frac{e^{2\xi\Im\t_*}d\vk^\a(x)}{ \fc(\a-\eta x) }.
\end{align}
Passing to the limit in \eqref{13oct1} as $\l\to 0$, we get 
\begin{equation*}
\lim_{\l\to 0}\log \frac{\fc(\a-\eta x)v_{\a-\eta x}(\l)}{ \fc(\a)v_{\a}(\l)}=x \Im\t_*  -
\frac{\sqrt{\l_*}} i\int_0^ x \frac{e^{2\xi\Im\t_*}d\vk^\a(\xi)}{\fc(\a-\eta \xi) }.
\end{equation*}
Using \eqref{13oct2}, we obtain
\begin{align*}
\lim_{\l\to -0}\log \frac{v_{\a-\eta x}(\l)}{v_{\a}(\l)}&=x\Im\t_*  -
\frac{\sqrt{\l_*}} i\int_0^ x \frac{e^{2\xi\Im\t_*}d\vk^{\a+\fj}(\xi)}{\fc(\a+\fj-\eta \xi) }
\\
&=x\Im\t_*  +
\int_0^ x \frac{d e^{-2\xi \Im\t_*}k^{\a+\fj-\eta\xi}(\l_*,\l_*)}{e^{-2\xi\Im\t_*}(k^{\a-\eta\xi}(\l_*,\l_*)+k^{\a+\fj-\eta \xi}(\l_*,\l_*) )},
\end{align*}
which we can bring to a more symmetric form
\begin{align*}
&=x\Im\t_*+\frac 1 2 \log {e^{-2\xi\Im\t_*}(k^{\a-\eta\xi}(\l_*,\l_*)+k^{\a+\fj-\eta \xi}(\l_*,\l_*) )}|_0^x
\\
&-\frac 1 2\int_0^ x \frac{d\, e^{-2\xi\Im\t_*}(k^{\a-\eta\xi}(\l_*,\l_*)-k^{\a+\fj-\eta \xi}(\l_*,\l_*) )}{e^{-2\xi\Im\t_*}(k^{\a-\eta\xi}(\l_*,\l_*)+k^{\a+\fj-\eta \xi}(\l_*,\l_*) )}
\\
&=\log\sqrt{\frac{k^{\a-\eta x}(\l_*,\l_*)+k^{\a+\fj-\eta x}(\l_*,\l_*) }{k^{\a}(\l_*,\l_*)+k^{\a+\fj}(\l_*,\l_*) }}
\\
&+\frac 1 2\int_0^ x \frac{d\, e^{-2\xi\Im\t_*}(k^{\a+\fj-\eta\xi}(\l_*,\l_*)-k^{\a-\eta \xi}(\l_*,\l_*) )}{e^{-2\xi\Im\t_*}(k^{\a-\eta\xi}(\l_*,\l_*)+k^{\a+\fj-\eta \xi}(\l_*,\l_*) )}.
\end{align*}
Thus, \eqref{15oct10}, \eqref{28sep1t} are proved.

The second limit \eqref{16oct1} is an essentially more delicate question. In fact, we again pass to the limit in \eqref{13oct1} and claim that
(for any $\a\in\pi_1(\O)^*$)
$$
\lim_{\l\to-\infty}\frac{\log v_{\a}(\l)}{\t(\l)}=0.
$$
But, this is exactly the claim of  Theorem 2 and Theorem 3 in \cite{VoYu16}.
\end{proof}


\begin{lemma}\label{l15oct1} Let
$$
\fe_{\a}(x)=\exp
\frac 1 2\int_0^ x \frac{d\, e^{-2\xi\Im\t_*}(k^{\a+\fj-\eta\xi}(\l_*,\l_*)-k^{\a-\eta \xi}(\l_*,\l_*) )}{e^{-2\xi\Im\t_*}(k^{\a+\fj-\eta \xi}(\l_*,\l_*) 
+k^{\a-\eta\xi}(\l_*,\l_*))}
$$
and 
\begin{equation}\label{15octb4}
\ff_\a(\l,x)=\fe_{\a}(x)\sqrt{\fc(\a-\eta x)}
v_{\a-\eta x}(\l)e^{ix\t(\l)}.
\end{equation}
Then
\begin{equation}\label{15octb2}
d\ff_\a(\l,x)=
i\sqrt{\l} \ff_{\a+\fj}(\l,x)
d\tau^{\a+\fj}(x), \quad d\tau^{\a+\fj}(x)=\frac i{\sqrt{\l_*}}
\frac{e^{2x\Im\t_*}d\vk^{\a+\fj}(x)}{\U^{\a+\fj}(x)^2}.
\end{equation}
\end{lemma}

\begin{proof}
We define $a_{\a}(x)$ by
\begin{equation}\label{15octb1}
i \fc(\a-\eta x)e^{ix(\t_*-\bar \t_*)}d a_{\a}(x)
=a_{\a}(x)\sqrt{\l_*}   d\vk^{\a}(\xi).
\end{equation}
With this multiplier, 
 we get
\begin{align*}
 d\left(a_{\a}(x)
i \fc(\a-\eta x)v_{\a-\eta x}(\l)e^{ix(\t(\l)-\bar \t_*)}\right)\ \\ \nonumber
=
i \fc(\a-\eta x)v_{\a-\eta x}(\l)e^{ix(\t(\l)-\bar \t_*)}da_{\a}(x)
+
 a_{\a}(x)d\left(
i\fc(\a-\eta x)v_{\a-\eta x}(\l)e^{ix(\t(\l)-\bar \t_*)}\right),
\end{align*}
and by \eqref{b13apr1},
\begin{align}\label{14oct2}\nonumber
\nonumber
=
i \fc(\a-\eta x)v_{\a-\eta x}(\l)e^{ix(\t(\l)-\bar \t_*)}da_{\a}(x)
-a_{\a}(x)\sqrt{\l_*} e^{ix(\t(\l)-\t_*)} {v_{\a-\eta x}(\l)}  d\vk^{\a}(x)\ \\
\nonumber -a_{\a}(x)\sqrt{\l} e^{ix(\t(\l)-\t_*)} {v_{\a+\fj-\eta x}(\l)}  d\vk^{\a+\fj}(x)\ \\
=-a_{\a}(x)\sqrt{\l} e^{ix(\t(\l)-\t_*)} {v_{\a+\fj-\eta x}(\l)}  d\vk^{\a+\fj}(x).
\end{align}
We integrate \eqref{15octb1}
\begin{align*}
\log{a_{\a}(x)}=&\frac{\sqrt{\l_*}} i \int_0^x  \frac{e^{2x\Im\t_*}d\vk^{\a}(\xi)}{ \fc(\a-\eta x)}
\\
=&x\Im\t_*-\log\sqrt{\frac{k^{\a-\eta x}(\l_*,\l_*)+k^{\a+\fj-\eta x}(\l_*,\l_*) }{k^{\a}(\l_*,\l_*)+k^{\a+\fj}(\l_*,\l_*) }}
\\
+&\frac 1 2\int_0^ x \frac{d\, e^{-2\xi\Im\t_*}(k^{\a+\fj-\eta\xi}(\l_*,\l_*)-k^{\a-\eta \xi}(\l_*,\l_*) )}{e^{-2\xi\Im\t_*}(k^{\a-\eta\xi}(\l_*,\l_*)+k^{\a+\fj-\eta \xi}(\l_*,\l_*) )},
\end{align*}
to obtain explicitly 
$$
a_{\a}(x)=e^{x\Im\t_*}\left(\frac{ \fc(\a-\eta x)}{ \fc(\a)}\right)^{-1/2}\fe_{\a}(x).
$$
With this expression, going back to \eqref{14oct2}, we have 
$$
d\left(\fe_\a(x) \sqrt{\fc(\a-\eta x)} v_{\a-\eta x}(\l)e^{ix\t(\l)} \right)=
i{\sqrt{\l}} \fe_\a(x)v_{\a+\fj-\eta x}(\l) e^{ix\t(\l)}
\frac{e^{2x\Im\t_*}d\vk^{\a+\fj}(x)}{\sqrt{{\fc(\a-\eta x)}}}.
$$
Now, we recall  that
$$
\frac {\sqrt{\l_*}} i\U^\a(x)^2=\fe_{\a}(x)^2 {\fc(\a-\eta x)},
$$
as it was defined in \eqref{28sep1t}, see also \eqref{13oct0}.
Thus, using the symmetry properties 
$$
\fe_{\a+\fj}(x)=\frac 1{\fe_{\a}(x)},\quad \fc(\a+\fj)=\fc(\a),
$$
we obtain \eqref{15octb2}.

\end{proof}

\begin{proof}[Proof of Theorem \ref{thm2}]
We note that
$$
\fe_{\a}(x)\sqrt{\fc(\a-\eta x)}
v_{\a-\eta x}(\l)=\sqrt{ \cC(\a-\eta x)}\fe_{\a}(x)\sqrt{\frac{V_{\a+\fj-\eta x}(\l_*)}{V_{\a-\eta x}(\l_*)}}V_{\a-\eta x}(\l).
$$
According to this remark, 
we modify slightly the transfer matrix $\cA_{\a}(\l,x)$, see \eqref{27apr4}, making rescaling of the basis vectors in the scaling subspaces
\begin{align}\label{15oct1}\nonumber
\begin{bmatrix} e^{ix\t(\l)}&0\\ 0&e^{-ix\t(\l)}
\end{bmatrix}
\cV_{\a-\eta x}(\l)\begin{bmatrix}\fe_{\a+\fj}(x)\sqrt\frac{V_{\a-\eta x}(\l_*)}{V_{\a+\fj-\eta x}(\l_*)}&0\\
0&\fe_{\a}(x)\sqrt{\frac{V_{\a+\fj-\eta x}(\l_*)}{V_{\a-\eta x}(\l_*)}}
\end{bmatrix}\\
=\cV_{\a}(\l)\begin{bmatrix}\sqrt\frac{V_{\a}(\l_*)}{V_{\a+\fj}(\l_*)}&0\\
0&\sqrt{\frac{V_{\a+\fj}(\l_*)}{V_{\a}(\l_*)}}
\end{bmatrix}\fA_\a(\l,x).
\end{align}
In this normalization we have
$$
\begin{bmatrix}i\sqrt{\l}\ff_{\a+\fj}(\l,x)&\ff_{\a}(\l,x)\end{bmatrix}
=\begin{bmatrix}i\sqrt{\l}\ff_{\a+\fj}(\l,0)&\ff_{\a}(\l,0)\end{bmatrix}\fA(\l,x).
$$
Due to Lemma \ref{l15oct1},
\begin{align*}
d\begin{bmatrix}i\sqrt{\l}\ff_{\a+\fj}(\l,x)&\ff_{\a}(\l,x)\end{bmatrix}\cJ=&
\begin{bmatrix}-\l \ff_{\a}(\l,x)d\tau^{\a}(x)&i\sqrt{\l}\ff_{\a+\fj}(\l,x)d\tau^{\a+\fj}(x)
\end{bmatrix}\begin{bmatrix}0&1\\-1&0
\end{bmatrix}
\\
=&-\begin{bmatrix}i\sqrt{\l}\ff_{\a+\fj}(\l,x)&\ff_{\a}(\l,x)\end{bmatrix}\begin{bmatrix}d\tau^{\a+\fj}(x)&0\\0&\l d\tau^{\a}(x)
\end{bmatrix}.
\end{align*}
Therefore we get
$$
d\fA(\l,x)\cJ=-\fA(\l,x)\begin{bmatrix}d\tau^{\a+\fj}(x)&0\\0&\l d\tau^{\a}(x)
\end{bmatrix},
$$
that is, \eqref{15oct8}.

Respectively, for the new Weyl-Titchmarsh function we have
$$
\fm_+^\a(\l)=\frac{V_\a(\l_*)}{V_{\a+\fj}(\l_*)}m^\a_+(\l)=i\sqrt{\l}\frac{v_{\a+\fj}(\l)}{v_{\a}(\l)},
$$
which, according to \eqref{11oct2}, means \eqref{15oct9} and this is the last claim of Theorem \ref{thm2}. 
\end{proof}

\begin{proof}[Proof of Corollary \ref{c15oc1}]
To prove the first statement we use \eqref{15oct1}. We have
$$
\frac 1 C \|\cV_{\a}(\l)\|^{-1} \|\cV_{\a-\eta x}(\l)^{-1}\|^{-1}
e^{x\cM(\l)}\le \|\fA_\a(\l,x)\|\le C\|\cV_{\a-\eta x}(\l)\| \|\cV_{\a}(\l)^{-1}\| e^{x\cM(\l)}.
$$
By Theorems 2 and 3  \cite{VoYu16} for an arbitrary $\a\in\pi_1(\O)^*$ we have 
$$
\lim_{\l\to-\infty}\frac{\log\|\cV_{\a}(\l)^{\pm 1}\|}{\cM(\l)}=0.
$$ 
Therefore we get \eqref{15oc1}.

To prove  the second claim, let us rewrite \eqref{16oct1} into the form
\begin{equation*}
\lim_{\l\to -\infty}\frac{\sqrt{\l}}{\t(\l)} 
\frac 1 \ell\int_0^\ell
\frac{v_{\a+\fj-\eta x}(\l)}{v_{\a-\eta x}(\l)}\frac{e^{2x\Im\t_*}d\vk^{\a+\fj}(x)}{\fc(\a-\eta x) }
=1.
\end{equation*}
Recall that $\fc(\a)=\fc(\a+\fj)$. Let 
$$
d\k^\a(x)=\frac{e^{2x\Im\t_*}}{\fc(\a-\eta x)}(d\vk^\a(x)+d\vk^{\a+\fj}(x))
$$ and we define the densities 
$$
\rho^{\a}(x)=\frac{e^{2 x\Im\t_*}d\vk^\a(x)}{\fc(\a-\eta x)d\k^\a(x)}, \quad
\rho^{\a+\fj}(x)=\frac{e^{2 x\Im\t_*}d\vk^{\a+\fj}(x)}{\fc(\a+\fj-\eta x)d\k^\a(x)}.
$$ 
 Then, we have
\begin{align*}
1=\lim_{\l\to -\infty}\frac{\sqrt{\l}}{\t(\l)}\frac 1 {2\ell}\int_0^\ell\left(
\frac{v_{\a+\fj-\eta x}(\l)}{v_{\a-\eta x}(\l)}\rho^{\a+\fj}(x)
+\frac{v_{\a-\eta x}(\l)}{v_{\a+\fj-\eta x}(\l)}{\rho^\a(x)}\right)d\k^\a(x)
\\ 
\ge \lim_{\l\to -\infty}\frac{\sqrt{\l}}{\t(\l)}\frac 1 \ell\int_0^\ell
{\sqrt{\rho^\a(x)\rho^{\a+\fj}(x)}}d\k^\a(x).
\end{align*}
By \eqref{b13apr3}
$$
\frac 1 \ell\int_0^\ell {\sqrt{\rho^\a(x) \rho^{\a+\fj}(x)}}d\k^\a(x)=0
\quad\text{and}
\quad
\rho^\a(x) \rho^{\a+\fj}(x)=0 \text{ for a.e. }\ x \ \text{w.r.t.} \ \k^\a,
$$
in other words the measures are mutually singular.

\end{proof}

\section*{Acknowledgment}

The author  would like to express his gratitude to
Benjamin Eichinger and Roman Bessonov for very helpful discussions.
This work was supported by the Austrian Science Fund FWF, project no: P29363-N32.

\bibliographystyle{amsplain}
\providecommand{\MR}[1]{}
\providecommand{\bysame}{\leavevmode\hbox to3em{\hrulefill}\thinspace}
\providecommand{\MR}{\relax\ifhmode\unskip\space\fi MR }
\providecommand{\MRhref}[2]{%
  \href{http://www.ams.org/mathscinet-getitem?mr=#1}{#2}
}
\providecommand{\href}[2]{#2}

\bigskip

P. Yuditskii, Abteilung f\"ur Dynamische Systeme und Approximationstheorie, Institut f\"ur Analysis, Johannes Kepler Universit\"at Linz, A-4040 Linz, Austria

\emph{E-mail address:} {Petro.Yudytskiy@jku.at}

\end{document}